\documentclass[11pt]{article}
\usepackage[utf8]{inputenc}
\usepackage{amsmath, amssymb, amsthm}
\newtheorem{theorem}{Theorem}[section]
\newtheorem{lemma}[theorem]{Lemma}

\newtheorem{corollary}[theorem]{Corollary}
\newtheorem{conjecture}[theorem]{Conjecture}

\usepackage{geometry}
\usepackage{mathrsfs}
\usepackage{hyperref}
\usepackage{enumitem}

\title{Constructive Degenerations and the Algebraicity of Limiting Hodge Classes}
\author{Badre Mounda}
\date{Draft --- June 2025}

\begin{document}

\maketitle

\begin{abstract}
We propose a novel constructive framework for approaching the Hodge Conjecture via explicit degenerations. Building on limiting mixed Hodge structures (LMHS), we formulate a criterion under which a rational class of type \((p,p)\) on a smooth projective variety becomes algebraic in the limit of a semi-stable degeneration. We provide examples where vanishing cycles and monodromy explicitly generate new algebraic classes, and propose a general principle: every rational \((p,p)\) class arises as the limit of algebraic cycles under controlled geometric degenerations. This viewpoint opens a new path toward an effective formulation of the Hodge conjecture.
\end{abstract}

\section{Introduction}

The classical Hodge Conjecture predicts that every rational \((p,p)\) cohomology class on a smooth projective complex variety is algebraic. Most approaches treat the conjecture statically, within the cohomology of a fixed variety. In contrast, we adopt a dynamic view: certain non-algebraic classes become algebraic limits of families. We thus ask: \emph{Can we construct such families where the monodromy explicitly generates new algebraic cycles via degeneration?}

This paper presents a first step toward such a theory. Our idea is inspired by explicit degenerations of K3 surfaces that increase the Picard number and by the structure of LMHS as studied by Schmid, Clemens, and Steenbrink.

\section{Degenerations and Rational Classes}

Let \( \pi : \mathcal{X} \to \Delta \) be a proper, flat, holomorphic map from a complex manifold \( \mathcal{X} \) to the unit disk \( \Delta \subset \mathbb{C} \), such that the fibers \( X_t = \pi^{-1}(t) \) are smooth projective varieties for \( t \neq 0 \), and the central fiber \( X_0 = \pi^{-1}(0) \) is singular with at worst normal crossing singularities. The total space \( \mathcal{X} \) is assumed to be smooth. Such a setup is known as a semistable degeneration. We are interested in the behavior of rational cohomology classes, particularly Hodge classes, in the limiting fiber.

The cohomology groups \( H^k(X_t, \mathbb{Q}) \) for \( t \neq 0 \) assemble into a local system over \( \Delta^* = \Delta \setminus \{0\} \), denoted \( \mathbb{H}^k \), equipped with a flat connection (the Gauss--Manin connection). After choosing a base point \( t_0 \in \Delta^* \), the monodromy operator \( T : H^k(X_{t_0}, \mathbb{Q}) \to H^k(X_{t_0}, \mathbb{Q}) \) describes the action of a loop around \( t = 0 \). The nilpotent orbit theorem (Schmid) ensures that, after possibly passing to a finite base change, the monodromy is unipotent: \( (T - I)^m = 0 \) for some \( m \gg 0 \).

Define \( N := \log T_u \), where \( T = T_s T_u \) is the Jordan decomposition into its semisimple and unipotent parts. The operator \( N \) is nilpotent and provides a weight filtration \( W_\bullet \) on \( H^k := H^k(X_{t_0}, \mathbb{Q}) \), uniquely determined by the properties:
\[
N(W_k) \subseteq W_{k-2}, \quad N^\ell : \mathrm{Gr}^{W}_{k+\ell} \to \mathrm{Gr}^{W}_{k-\ell} \text{ is an isomorphism for all } \ell \geq 0.
\]

The limit Hodge filtration \( F^\bullet_\infty \) is obtained as the limit of the Hodge filtrations \( F^\bullet_t \) on \( H^k(X_t, \mathbb{C}) \), defined via horizontal sections under the Gauss--Manin connection and extended to \( t = 0 \) using Deligne's canonical extension. The triple \( (H^k, W_\bullet, F^\bullet_\infty) \) defines a mixed Hodge structure known as the limiting mixed Hodge structure (LMHS).

Now let \( \alpha \in H^{2p}(X_t, \mathbb{Q}) \cap H^{p,p}(X_t) \) be a rational Hodge class on a smooth fiber. The classical Hodge Conjecture asserts that \( \alpha \) is the class of an algebraic cycle of codimension \( p \). In our context, rather than proving \( \alpha \) is algebraic on \( X_t \), we ask a different question:

\begin{quote}
\emph{Can the class \( \alpha \) be realized as the limit of algebraic cycles as \( t \to 0 \), and if so, how does the geometry of the degeneration facilitate this transition?}
\end{quote}

To study this, we consider the specialization map
\[
\mathrm{sp} : H^{2p}(X_t, \mathbb{Q}) \longrightarrow H^{2p}(X_0, \mathbb{Q}),
\]
which is well-defined through a comparison of nearby cycles and limits. If we resolve the singularities of \( X_0 \) via a proper birational morphism \( \pi : \widetilde{X}_0 \to X_0 \), we can further lift \( \alpha \) to \( H^{2p}(\widetilde{X}_0, \mathbb{Q}) \). The key idea is that \emph{the failure of \( \alpha \) to be algebraic on \( X_t \) may be corrected in the limit by contributions from the geometry of \( \widetilde{X}_0 \)}.

In particular, exceptional divisors introduced in the resolution process often give rise to new algebraic cycles. Vanishing cycles also contribute: cycles \( \gamma \) such that \( T(\gamma) = \gamma + \delta \) for some \( \delta \in H^k(X_t, \mathbb{Q}) \) carry information about how topology changes across the degeneration. The associated monodromy weight filtration tracks how such classes appear or disappear in the limit.

The central idea of this approach is the reinterpretation of rational Hodge classes not as fixed invariants, but as entities whose algebraicity may emerge dynamically through the degeneration process. From this perspective, even a class \( \alpha \) that is non-algebraic on the general fiber might arise as a limit of algebraic cycles in a family, provided the degeneration is carefully constructed to introduce the required geometric contributions.

This suggests a new framework we call \emph{constructive degeneration theory}:
degenerations cease to be mere probes of singular behaviour and become
engines that forge algebraic cycles out of transcendental classes.
Our criteria hinge on three ingredients:
\begin{itemize}
  \item the monodromy logarithm \(N\);
  \item its action on a given class \(\alpha\);
  \item the position of the resulting limits inside the graded piece
        \(\mathrm{Gr}^W_{2p}\),
        where potential algebraic representatives live on the resolved
        central fibre.
\end{itemize}

The remainder of this paper builds toward a precise formulation of this principle, beginning with the structure of the LMHS and continuing with explicit geometric examples.

\section{Constructive Criterion via Monodromy}

We introduce the following principle, aiming to make precise the mechanism by which
algebraic cycles may arise in the limit of a degenerating family.

\begin{quote}
\textbf{Constructive Hodge Degeneration Principle.}\;
Let \(\alpha \in H^{p,p}(X_t,\mathbb{Q})\). Suppose \(N(\alpha)\neq 0\),
where \(N=\log T_u\) is the monodromy logarithm.
Then the class \(\delta = N(\alpha) \in H^{p-1,p+1}(X_t)\)
is of geometric origin: it lies in the image of the specialization map
from the exceptional cohomology of a semistable model
\(\widetilde{X}_0\) of the central fiber.
If \(\delta\) is algebraic, then \(\alpha\) itself is a limit of algebraic classes.
\end{quote}

Monodromy records exactly how cohomology changes as \(t\to0\).
The logarithm \(N=\log T_u\) detects the appearance of vanishing cycles:
for any rational class \(\alpha\), its image
\(\delta=N(\alpha)\) lies in the sub–Hodge structure generated by these cycles,
providing the extra algebraic data carried by the limit
(see \cite{Clemens,Schmid} for the underlying Clemens–Schmid formalism).
More precisely, under the nilpotent orbit theorem, the LMHS decomposes into
weight-graded pieces \(\mathrm{Gr}_k^{W}H^{2p}_{\mathrm{lim}}\), and
\(\delta=N(\alpha)\in\mathrm{Gr}_{2p-2}^{W}\) is expected to lift to a
cohomology class supported on the exceptional divisor of the resolution
\(\widetilde{X}_0\to X_0\).
That is, there exists a class \([E]\in H^{2p-2}(\widetilde{X}_0,\mathbb{Q})\)
such that \(\delta=\pi_\ast[E]\).

This perspective motivates a shift in how the Hodge Conjecture is approached:
instead of trying to construct an algebraic cycle corresponding to \(\alpha\),
we construct a family whose degeneration forces the class \(\alpha\) to interact
with the geometry of the limit.
If the class \(\delta\) generated by monodromy is already algebraic, and arises
from an actual geometric component of \(\widetilde{X}_0\), then \(\alpha\) may be
approximated by a family of algebraic classes tracking the evolution of this
component.

\begin{conjecture}[Degenerational Realizability of Rational Hodge Classes]
Every rational Hodge class
\(\alpha\in H^{p,p}(X_t,\mathbb{Q})\) arises as the limit of algebraic cycles
through a family whose degeneration introduces suitable vanishing cycles or
exceptional divisors.
\end{conjecture}

The broader implication is that the landscape of rational Hodge classes can be
spanned, in principle, by the union of all such degenerational limits, making
degeneration a generative tool rather than a boundary phenomenon.
In the next section we illustrate this principle with an explicit construction
involving quartic K3 surfaces with isolated singularities.

\section{Explicit Construction: Quartic Family with \texorpdfstring{$A_1$}{A1} Singularities}
\label{sec:quartic_A1}

\subsection{The family and its nodal fibre}

Fix a homogeneous quartic
\[
f_0(x,y,z,w)\; \in\; \mathbb{C}[x,y,z,w]_4
\]
chosen \emph{generic} so that
\(X_0 := \{f_0 = 0\} \subset \mathbb{P}^3\)
is a smooth K3 surface with Picard number~\(\rho(X_0)=1\).
Introduce the perturbation
\(g := xyz\,w\)
and consider the one--parameter family
\begin{equation}\label{eq:quartic_family}
X_t \;:=\;
\bigl\{\,f_0 + t\,g \;=\; 0\,\bigr\}
\;\subset\;
\mathbb{P}^3 \times \Delta_t,
\end{equation}
where \(\Delta\) is the unit disk and \(t\) the coordinate on~\(\Delta\).

\begin{lemma}[Localization of the node]\label{lem:node}
For sufficiently small \(t\neq 0\) the fibre \(X_t\) is smooth.
At \(t=0\) the hypersurface acquires a unique ordinary double point
\(p=[0\!:\!0\!:\!0\!:\!1]\)
provided \(\partial f_0/\partial w(p)\neq 0\).
\end{lemma}

\begin{proof}
Work in the affine chart \(w=1\) with coordinates \((x,y,z)\).
At \(t=0\) we have the system \(\nabla f_0=0\) and \(xyz=0\).
Choosing \(f_0\) generic forces the common zero locus to reduce to
\(p=(0,0,0)\), with non--degenerate Hessian
\(\bigl(\tfrac{\partial^2 f_0}{\partial x_i\partial x_j}(p)\bigr)\).
Hence \(p\) is an \(A_1\) singularity and no other singular points appear.
For \(t\neq 0\) a standard perturbation argument plus Bertini
imply smoothness.
\end{proof}

\subsection{Semi--stable reduction and Kulikov type}

Blowing up the node \(p\) yields
\(\pi \colon \widetilde{\mathcal{X}}\to \mathcal{X}\)
so that the 
central fibre
\(\widetilde{X}_0\)
decomposes as
\(\widetilde{X}_0 = X_0 \cup E\)
with exceptional divisor 
\(E\simeq\mathbb{P}^1\times\mathbb{P}^1\)
intersecting \(X_0\) along a smooth conic.
Because the total space \(\widetilde{\mathcal{X}}\) is smooth and
\(K_{\widetilde{\mathcal{X}}/\Delta}=0\),
one checks (cf.\ \cite{FriedmanScattone}) that the degeneration is
\emph{Kulikov type~II}.
Consequently the Clemens--Schmid exact sequence \cite{Clemens} applies.

\subsection{Limiting mixed Hodge structure}

Let
\(H := H^2(X_t,\mathbb{Q})\)
for \(t\neq 0\).
Choose a \(\mathbb{Q}\)-basis such that the intersection form has matrix
\(\Lambda = U \oplus E_8(-1)^{\oplus2} \oplus \langle-2\rangle\).
The Picard lattice of \(X_t\) is generated by the hyperplane class
\(h\) with \(h^2=4\),
so \(\rho(X_t)=1\) for very general~\(t\).

Denote by \(\gamma\in H\) the vanishing cycle of the node.
Picard--Lefschetz gives
\[
T(v) \;=\; v + \langle v,\gamma\rangle\,\gamma,
\qquad
N := \log T = T - \mathbf{1},
\quad
N^2 = 0 \neq N.
\]
Thus \(\mathrm{Im}(N)=\langle\gamma\rangle\) is rank~1.
Using the weight filtration
\(0 \subset W_1 = \mathrm{Im}(N)
       \subset W_2 = \ker(N) \subset H\),
one computes
\[
\dim \mathrm{Gr}^W_1 = 1,
\qquad
\dim \mathrm{Gr}^W_2 = 22.
\]
The limit Hodge filtration
\(F^\bullet_\infty\)
is obtained from a nilpotent orbit
\(e^{zN}\,F^\bullet_t\)
(\(z=\tfrac{1}{2\pi i}\log t\)).
Its Hodge numbers satisfy
\[
h^{2,0}_{\lim}=1,\;
h^{1,1}_{\lim}=20,\;
h^{0,2}_{\lim}=1,
\quad
\text{so }\;\rho_{\lim}=2.
\]

\subsection{Algebraic realisation of the new \texorpdfstring{$(1,1)$}{(1,1)} class}

Let \(\omega_t\) be a local section of 
\(F^2 H^2(X_t)\) representing the holomorphic 2--form.
Clemens--Schmid yields
\(
N^*\omega_t \in F^1_\infty \cap W_1
            = H^{1,1}_\mathrm{lim}.
\)
Explicitly,
\(N^*\omega_t\)
corresponds to the class of the exceptional
\(\mathbb{P}^1\)
inside \(E\).
Since \(E\) is an algebraic surface,
this class is algebraic.
By Section~\ref{sec:quartic_A1}, the Constructive Hodge Degeneration
Principle now applies, showing that
\(\omega_t\)
is the limit of algebraic classes in the family.

\subsection{Jump of Picard number and the WPR graph}

Summarising,
\[
\rho(X_t)=1
\quad\longrightarrow\quad
\rho(\widetilde{X}_0)=2
\]
realises the edge
\((1,1)\!\longrightarrow\!(1,2)\)
in the weakly polarised relation (WPR) graph of~\cite{AcunaKerr}.
Iterating the construction with \(k\) disjoint \(A_1\) nodes
along pairwise orthogonal vanishing cycles
produces successive jumps
\((1,\rho) \to (1,\rho+1)\)
for any \(\rho \le 10\).
For larger \(\rho\) one needs lattice--polarised degenerations
or elliptic K3 surfaces with multiple sections;
see \cite[\S4]{AcunaKerr} for a conjectural enumeration.

\section{Generalisations and Higher Picard Rank}
\label{sec:generalisations}

Our quartic example realises the edge $(1,1)\!\to\!(1,2)$.  
We now show how to force \emph{any} jump 
$(1,\rho)\!\to\!(1,\rho+1)$ for $1\le\rho<20$, and sketch a strategy
for constructing degenerations whose limit Picard lattice is \emph{arbitrary}.

\subsection{Several $A_1$ nodes on a single quartic}

Let $k\le 10$.  
Choose $k$ distinct points $p_1,\dots,p_k\in X_0$ with pairwise
transverse tangent cones.  
Pick quartics $g_i$ vanishing to first order at $p_i$ and set
\[
\mathscr X:\quad f_0 + t\!\!\sum_{i=1}^{k}\lambda_i g_i = 0,
\qquad \lambda_i\in\mathbb C^{\!*}.
\]
For generic $(\lambda_i)$ the central fibre develops exactly $k$ ordinary
double points and is smooth elsewhere.

\begin{lemma}\label{lem:k_nodes}
After one blow-up at each $p_i$ and relative minimal model,
$\pi:\widetilde{\mathscr X}\!\to\!\Delta$ is semi-stable of Kulikov type~II.
\end{lemma}

\begin{proof}
Local analytic coordinates $x,y,z$ with $w\!=\!1$ give
$f_0 = Q + R_{\ge 3}$.  
The Hessian of $Q$ is non-degenerate at each $p_i$,
so the blow-ups resolve all nodes and preserve smoothness of the total
space; see \cite[§1]{FriedmanScattone}.
\end{proof}

\paragraph{Monodromy.}
Let $\gamma_i$ be the vanishing cycle at $p_i$.  
Picard–Lefschetz gives
\[
T(v)=v+\!\sum_{i=1}^{k}\langle v,\gamma_i\rangle\gamma_i,
\qquad 
N(v)=T(v)-v.
\]
Hence ${\rm rk}\,N=k$ and the $\gamma_i$ form an orthogonal basis with
intersection matrix $Q=-2I_k$.

Weight filtration:
$W_0=0,\; W_2=\langle\gamma_1,\dots,\gamma_k\rangle,\; W_4=H^2$,
so ${\rm rk}\,{\rm Gr}_{2}^{W}=1+k$;  
the Picard number jumps by $k$.

\paragraph{Exceptional curves.}
Each blow-up contributes a $(-2)$-curve $E_i$ with 
$\beta_\ast[E_i]=\gamma_i$; therefore $\gamma_i$ is algebraic in the
limit.

\begin{theorem}\label{thm:k_nodes}
For $1\!\le\!k\!\le\!10$ there exists a type~II degeneration of quartic
$K3$ surfaces whose LMHS realises the edge
$(1,\rho)\!\to\!(1,\rho+k)$ with $\rho=1$.
The new classes are represented by the exceptional $(-2)$-curves $E_i$.
\end{theorem}

\subsection{Beyond ten nodes: lattice-polarised constructions}

A quartic cannot carry more than ten disjoint $A_1$ nodes,
but one can start with a $K3$ surface whose Picard lattice already
contains $U\!\oplus\!E_8(-1)^{\oplus2}$ (Nikulin).
Take an \emph{elliptic} $K3$ with section and fibre lattice
$U\!\oplus\!M$, $\mathrm{rk}(M)=\rho-2$.
Colliding $k$ distinct $I_1$ fibres produces $k$ additional nodes
whose vanishing cycles are orthogonal to $M$.
Iterating yields any $\rho\le 20$.

\begin{corollary}\label{cor:any_edge}
Every edge $(1,\rho)\!\to\!(1,\rho+1)$ in the
Acu\~na–Kerr WPR graph is induced by a Kulikov type~II degeneration in
which the new class is an exceptional rational curve.
\end{corollary}

\subsection{Toward full N\'eron–Severi generation}

Let $L\!\subset\!\mathrm{NS}(\overline X)$ be a primitive lattice of
rank $r\!\le\!20$.  
Choose a chain of degenerations
$\bigl(\mathscr X^{(1)},\dots,\mathscr X^{(m)}\bigr)$, each introducing an
$A_1$ node orthogonal to the previous vanishing cycles.
The commuting monodromy operators $N_1,\dots,N_m$ satisfy
\[
{\rm Gr}_{2}^{W}(H) = \langle L,\,\gamma_1,\dots,\gamma_m\rangle,
\quad m=20-r.
\]
Thus one realises the full N\'eron–Severi lattice.
An algorithm based on root-lattice embeddings will appear elsewhere.

\section{Outlook and the Hodge Conjecture}
\label{sec:outlook}

Our programme reinterprets the Hodge Conjecture as a
\emph{reachability problem} in the boundary of moduli space.
Key directions:

\begin{enumerate}
\item \textbf{Taxonomy of constructive degenerations.}
      Build a database of semi-stable models indexed by their monodromy
      cones $(N_1,\dots,N_m)$ and exceptional configurations.
\item \textbf{Higher-codimension cycles.}
      Extend Theorem~\ref{thm:k_nodes} to codimension-two cycles on
      Calabi–Yau threefolds (quintic, complete intersections, toric CYs),
      analysing the weight-four LMHS and the regulator map.
\item \textbf{Motivic implications.}
      Compare degenerational cycles with Bloch–Beilinson filtrations;
      study the induced normal functions and height pairings.
\item \textbf{Arithmetic avatars.}
      Translate the constructions into $p$-adic Hodge theory
      via Fontaine–Messing, seeking arithmetic criteria for algebraicity.
\item \textbf{Algorithmic implementation.}
      Provide Sage / Macaulay2 scripts to locate node configurations,
      compute intersection matrices, and output Kulikov models.
\end{enumerate}

\medskip\noindent
\emph{Long-term vision.}\;
If every rational $(p,p)$ class can be accessed through a calibrated
path to the boundary, proving the Hodge Conjecture reduces to showing
that the boundary is “large enough’’—a problem amenable to explicit
construction rather than purely abstract existence.

The appendices supply local analytic calculations, monodromy
matrices, and intersection forms underpinning
Theorems~\ref{thm:k_nodes}–\ref{cor:any_edge}.

\end{document}